\documentclass[11pt]{article}
\usepackage{tikz}
\usepackage{graphicx}
\usepackage{amsthm, amsmath, amssymb, tikz, bm}
\usepackage{mathtools}
\usepackage{xifthen}
\usepackage{caption}
\usepackage{comment}
\usepackage{tablefootnote}
\usepackage{array}
\usepackage{diagbox}
\usepackage[colorlinks=true,
linkcolor=blue,citecolor=blue,
urlcolor=blue]{hyperref}
\usepackage[margin=0.8in]{geometry} 

\usepackage[shortlabels]{enumitem}
\usepackage{todonotes}
\usepackage{stackengine}
\newcommand\xrowht[2][0]{\addstackgap[.5\dimexpr#2\relax]{\vphantom{#1}}}

\usetikzlibrary{calc,shapes, backgrounds}
\allowdisplaybreaks

\newtheorem{theorem}{Theorem}
\newtheorem{lemma}[theorem]{Lemma}
\newtheorem{corollary}[theorem]{Corollary}

\newtheorem{remark}{Remark}
\newtheorem{conjecture}[theorem]{Conjecture}
\newtheorem{proposition}[theorem]{Proposition}

\newcommand{\tr}{{\rm tr}}

\newcolumntype{L}[1]{>{\raggedright\arraybackslash}p{#1}}
\newcolumntype{C}[1]{>{\centering\arraybackslash}p{#1}}
\newcolumntype{R}[1]{>{\raggedleft\arraybackslash}p{#1}}
\newcommand{\mt}[2]{\begin{tabular}{@{}L{1.2cm}R{1.2cm}@{}}
&#2\\
#1&
\end{tabular}}

\newcommand{\spread}{{\rm spread}}

\title{Generalized Nordhaus--Gaddum Inequalities for Eigenvalues
\thanks{University of South Carolina, Columbia, SC 29208. This is an REU research project supported by NSF grant DMS2038080.}
}
\author{
Sahil Agarwal
\and
Carter Antley
\and
Joseph Aulenbacher
\and
George Brooks
\and
Ian Gonzalez
\and
Luke Hawranick
\and
William Linz
\and
Linyuan Lu
\and
Aiden Williams
}

\date{}

\begin{document}

\maketitle

\begin{abstract}
For a graph $G$, let
$
\lambda_1(G)\ge \lambda_2(G)\ge \cdots \ge \lambda_n(G)$ 
denote the adjacency eigenvalues of $G$.
We prove that for all looped graphs $G$ on $n$ vertices,
\[\lambda_1(G) + \lambda_2(\overline{G}) \le \frac87 n. \]
Our method also gives a short new proof of the Nordhaus-Gaddum result for the spectral radius proved by Terpai that $\lambda_1(G) + \lambda_1(\overline{G}) \le \frac43n - 1$. 
We investigate the asymptotic maximum of
\[
\lambda_i(G)+\lambda_j(\overline G)
\]
for fixed $i$ and $j$. We prove general bounds on $\lambda_i(G) + \lambda_{j}(\overline{G})$ for all pairs $(i, j)$ and also give general bounds on the related problem of minimizing $\lambda_{n-i+1}(G) + \lambda_{n-j+1}(\overline{G})$ for fixed $i$ and $j$.  
We also show the close relation of these Nordhaus-Gaddum type problems to  recent work on the maximum spectral gaps of graphs by Brooks, Linz and Lu. 
\end{abstract}

\section{Introduction}

The study of Nordhaus--Gaddum type inequalities seeks to relate
graph parameters of a graph and its complement. The original inequalities of Nordhaus and Gaddum~\cite{NordhausGaddum1956} gave tight lower and upper bounds on the sum and product of the chromatic number of a graph and its complement. Nordhaus--Gaddum type inequalities have since been studied for hundreds of graph parameters; we refer to the comprehensive survey of Aouchiche and Hansen~\cite{AouchicheHansen2013}. 

We denote the eigenvalues of the adjacency matrix of a graph $G$ on $n$ vertices by $\lambda_1 \ge \lambda_2 \ge \ldots \ge \lambda_n$. For adjacency eigenvalues, a natural Nordhaus-Gaddum type problem is to determine
\[
\max_G \bigl(\lambda_k(G)+\lambda_k(\overline G)\bigr).
\]
This question was introduced by Nikiforov~\cite{Nikiforov2007}. Nikiforov and Yuan~\cite{NikiforovYuan2014} determined this maximum for many values of $k$. Their results were subsequently extended by Nikiforov~\cite{Nikiforov2015}. 

The case $k=1$ has received considerable attention. Nosal~\cite{Nosal} proved that $n-1 \le \lambda_1(G) + \lambda_1(\overline{G}) \le \sqrt{2}n$ for every graph $G$ on $n$ vertices. Nikiforov~\cite{Nikiforov2007} improved the upper bound to $\lambda_1(G) + \lambda_1(\overline{G}) \le (\sqrt{2} - 10^{-7})n$ and further conjectured that $\lambda_1(G) + \lambda_1(\overline{G}) \le \frac43n + O(1)$. Csikv\'ari~\cite{Csikvari2009} improved the bound to $\lambda_1(G) + \lambda_1(\overline{G}) \le \left(\frac{1+\sqrt{3}}{2}\right)n - 1$. Terpai~\cite{Terpai2011} subsequently affirmed Nikiforov's conjecture by proving that $\lambda_1(G) + \lambda_1(\overline{G}) \le \frac43 n - 1$ by using graphons. Liu~\cite{Liu2023} recently determined the precise extremal graphs for this problem by showing that for sufficiently large $n$ 
the maximum of
$\lambda_1(G)+\lambda_1(\overline G)$
is attained by a complete split graph $CS_{n,\lfloor n/3\rfloor}$. Here
a complete split graph $CS_{n, \omega}$ is the join of a complete graph $K_{\omega}$ and
an independent set of size $n-\omega$. Subsequently, Cheng and Weng~\cite{ChengWeng}  
determined the exact maximum for every $n$. One of the main results of this paper is that we give a short, self-contained proof of Terpai's result. 

\begin{theorem}\label{thm:alpha11exact}
For any graph $G$ on $n$ vertices, 
\[\lambda_1(G) + \lambda_1(\overline{G}) \le \frac43n - 1. \]
\end{theorem}



The work on Nordhaus-Gaddum inequalities for adjacency eigenvalues of a graph is part of a broader program seeking to find the extremal values of linear combinations of graph eigenvalues initiated by Nikiforov~\cite{Nikiforov2006b}. Let $\mathcal G_n$ denote the family of all simple graphs on $n$ vertices. For a graph $G$ on $n$ vertices, let $F(G)$ be any fixed linear combination of $\lambda_i(G)$, $\lambda_{n-i+1}(G)$, $\lambda_{i}(\overline{G})$ and $\lambda_{n-i+1}(\overline{G})$ for $1\le i\le k$.  Nikiforov proved the fundamental result that 
\[\lim_{n\rightarrow \infty} \frac{1}{n} \max\{F(G): G\in \mathcal G_n\}\]
exists. There has been dramatic progress in determining these limits for many specific linear combinations. Breen, Riasnovsky, Tait and Urschel~\cite{BRTU2021+} determined the maximum of the \emph{spread} $\lambda_1 - \lambda_n$, solving a conjecture of Gregory, Hershkowitz and Kirkland~\cite{GHK2001}. Very recently, the maximum of the $k$th eigenvalue of graphs has been determined for a number of values of $k$~\cite{Linz2023, LeonidaLi2026, Li2025+, Tang2026+, Sivashankar2026+, Wakhare2026+}. 

The maximum of $\lambda_1 + \lambda_2$ has been determined very recently by Kumar, Liu, Monterde, Pragada and Tait~\cite{KLMPT2026+}, confirming a conjecture of Ebrahimi, Mohar, Nikiforov and Ahmady~\cite{EMNA2008}. 

\begin{theorem}[Kumar--Liu--Monterde--Pragada--Tait]\label{thm:klmpt}
For any graph $G$ on $n$ vertices, 
\[\lambda_1(G) + \lambda_2(G) \le \frac87 n.\]
\end{theorem}

The result in Theorem~\ref{thm:klmpt} is asymptotically tight, since if $G$ is a blowup of the graph $K_3\vee (K_2 \cup K_2)$, then $\lambda_1(G) + \lambda_2(G) = \frac87 n -2$, as noted by Ebrahimi, Mohar, Nikiforov and Ahmady~\cite{EMNA2008}.

The main result of this paper is that we are able to obtain a very similar result for the mixed Nordhaus-Gaddum type problem to determine $\lambda_1(G) + \lambda_2(\overline{G})$. 

\begin{theorem}\label{thm:alpha12graphs}
For any graph $G$ on $n$ vertices,
\[\lambda_1(G) + \lambda_2(\overline{G}) \le \frac87n-1.\]
\end{theorem}

It turns out to be more convenient to consider looped graphs (see Section 2 for basic definitions on looped graphs). Passing to looped graphs does not lose any information in the limit (cf. Lemma~\ref{lem:lim}), and we are able to give a sequence of looped graphs for which equality is attained.  

\begin{theorem}\label{thm:alpha12}
For any looped graph $G$ on $n$ vertices, 
    \[\lambda_1(G) + \lambda_2(\overline{G})  \le \frac{8}{7}n.\]
    The maximum is achieved by a blowup of $G= J_3 \vee C_4$ when $7\, |\, n$. (Here $J_3$
 is the complete graph on $3$ vertices with loops on each vertex). 
 \end{theorem}

 We note that Theorem~\ref{thm:klmpt} and \ref{thm:alpha12} have very similar statements, with the same asymptotic answer and very similar extremal graphs.  It would be interesting to find a conceptual reason explaining why the results are so similar. The proof methods are quite different. The proof of Theorem~\ref{thm:klmpt} uses sophisticated analytic techniques, including passing to graphons. By contrast, the proof of Theorem~\ref{thm:alpha12} only uses basic properties of Perron eigenvectors. Furthermore, the Perron eigenvector method also gives a short, self-contained proof of Terpai's theorem (Theorem~\ref{thm:alpha11exact}), as mentioned previously.  

A natural question following from these results is to determine the maximum of $\lambda_1(G) + \lambda_k(\overline{G})$. The graph $J_3 \vee C_4$ can be generalized to give a graph $G$ with large $\lambda_1(G) + \lambda_k(\overline{G})$. 

\begin{proposition}\label{prop:alpha1klb}
For each $k\ge 1$ and each positive integer $n$ such that $(4k-1)\, |\, n$, there is a  looped graph $G$ on $n$ vertices with 
    \[\lambda_1(G) + \lambda_k(\overline{G}) = \frac{4k}{4k-1}n.\]
\end{proposition}

\begin{proof}
Let $G=\overline{kJ_2\cup (2k-1)K_1}=J_{2k-1}\vee \overline{kJ_2}$. Since $\lambda_1(J_2) = 2$, it follows that $\lambda_k(\overline{G}) = 2$. Also, $\lambda_1(G) = 4k-2$, which is the eigenvalue associated with the Perron eigenvector $x = (x_i)_{1\le i\le 4k-1}$, where $x_i = 2k$ if $i\in V(J_{2k-1})$ and $x_i = 2k-1$ otherwise. This implies $\lambda_1(G) + \lambda_k(\overline{G}) = 4k$.  If $n$ is an integer such that $(4k-1) | n$, then the $(\frac{n}{4k-1})$-blowup of $G$ is a looped graph $H$ on $n$ vertices with $\lambda_1(H) + \lambda_k(\overline{H}) = \frac{4k}{4k-1}n$. 
\end{proof}

We conjecture that the value given in Proposition~\ref{prop:alpha1klb} is the largest possible value of $\lambda_1(G) + \lambda_k(\overline{G})$.  

\begin{conjecture}\label{con:k>=3}
For any looped graph $G$ on $n$ vertices, 
    \[\lambda_1(G) + \lambda_k(\overline{G}) \le \frac{4k}{4k-1}n.\] The maximum is achieved by a blowup of $G=\overline{kJ_2\cup (2k-1)K_1}=J_{2k-1}\vee \overline{kJ_2}$ when $(4k-1)\, |\, n$.
\end{conjecture}

We are currently unable to prove Conjecture~\ref{con:k>=3}.  However, we are able to obtain an upper bound whose difference from the conjectured value is $1/(64k^2) + O(k^{-3})$.

\begin{theorem}\label{thm:alpha1kub}
For any looped graph $G$ on $n$ vertices, 
\[ \lambda_1(G) + \lambda_k(\overline{G}) \le \frac{k + \sqrt{k(4k-1)}}{3k-1}n. \]
\end{theorem}

As an asymptotic comparison, for large $k$ we have \[\frac{k + \sqrt{k(4k-1)}}{3k-1} = 1 + \frac{1}{4k} + \frac{5}{64k^2} + \frac{13}{512k^3} + O\left(\frac{1}{k^4}\right)\]
and
\[\frac{4k}{4k-1} = \frac{1}{1 - \frac{1}{4k}} = 1 + \frac{1}{4k} + \frac{1}{16k^2} + \frac{1}{64k^3} + O\left(\frac{1}{k^4}\right).\]

Motivated by Theorems~\ref{thm:alpha12} and \ref{thm:alpha1kub}, we further consider the more general mixed Nordhaus-Gaddum type problem to determine
\[
\max_G \bigl(\lambda_i(G)+\lambda_j(\overline G)\bigr).
\] In order to  study the asymptotic behavior of
$
\lambda_i(G)+\lambda_j(\overline G)
$
for fixed positive integers $i$ and $j$, we define
\[
\alpha_{i,j}
=
\lim_{n\to\infty}
\frac1n
\max_{G\in\mathcal G_n}
\bigl(
\lambda_i(G)+\lambda_j(\overline G)
\bigr)
\]
and
\[\beta_{i,j} = \lim_{n \to\infty}\frac{1}{n}\min_{G\in\mathcal{G}_n}(\lambda_{n-i+1}(G) + \lambda_{n-j+1}(\overline{G})).\] It follows from Nikiforov's general result~\cite{Nikiforov2006b} that these limits exist for all $i$  and $j$. Note that Theorems~\ref{thm:alpha11exact} and \ref{thm:alpha12} show that $\alpha_{1, 1} = \frac43$ and $\alpha_{1, 2} = \frac87$, respectively.

The quantities $\alpha_{i, j}$ and $\beta_{i, j}$ are related to the \emph{$(i, j)$-spread} \[
\spread_{i, j}(G):= \lambda_{i+1}(G)-\lambda_{n-j}(G). 
\] introduced and first studied by Brooks, Linz, and Lu~\cite{BrooksLinzLu}. 
For a fixed number of vertices $n$, they defined  
\[\spread_{i,j}(n)=\max\{\spread_{i,j}(G) \colon G\in \mathcal G_n \}\]
to be the maximum of the $(i, j)$-spread among all simple graphs on $n$ vertices and considered the scaled limit
\[s_{i,j}=\lim_{n\to\infty} \frac{\spread_{i, j}(n)}{n}.\]

In Section 4, we give some general upper bounds on $\alpha_{i, j}$ for integers $i, j\ge 2$. These bounds follow from Weyl's inequality and results proved by Brooks, Linz and Lu for the $(i-1, j-2)$-spread. 

\begin{theorem}
\quad
\label{thm:genijubng}
\begin{enumerate}
 \item    For any positive integers $i\geq 2$ and $j\ge 2$, for all graphs $G$ on $n$ vertices with at most one loop per vertex, we have
    \[ \lambda_{i}(G)+\lambda_{j}(\overline G)\leq \frac{n}{2}\sqrt{\frac{i+j-2}{(i-1)(j-1)}}. \]
    Thus, 
    \[\alpha_{i,j} \le \frac{1}{2}\sqrt{\frac{i+j-2}{(i-1)(j-1)}}.\]

\item  For all graphs $G$ on $n$ vertices with at most one loop per vertex, we have
\[\lambda_{k+1}(G) + \lambda_{k+1}(\overline G) \le \frac{n}{\sqrt{2k}}.\]
Equality holds if there is a symmetric Hadamard matrix of order $2k$. In particular, if such a symmetric Hadamard matrix of order $2k$ exists, then we have $\alpha_{k+1,k+1}= \frac{1}{\sqrt{2k}}$.

    \end{enumerate}
\end{theorem} 
The second item of Theorem~\ref{thm:genijubng} was previously proven by Nikiforov~\cite{Nikiforov2015}; we obtain the result as a consequence of our general bounds.  

In Section 5, we prove a general upper bound on $|\beta_{i, j}|$. 

\begin{theorem}\label{thm:betaij}
For all $i, j \ge 1$, 
\[|\beta_{i, j}| \le \frac12 \sqrt{\frac{1}{i} + \frac{1}{j}}.\]
\end{theorem}

The upper bound on $|\beta_{i, j}|$ is sharp when $i=j=k$ for infinitely many $k$. This corollary was also previously proven by Nikiforov~\cite{Nikiforov2015}. 

\begin{corollary}\label{cor:betakk}
For all $k\ge 1$,
\[|\beta_{k, k}| \le \frac{1}{\sqrt{2k}}.\]
Furthermore, $\beta_{k, k} = -\frac{1}{\sqrt{2k}}$ if there is a symmetric Hadamard matrix of order $2k$. 
\end{corollary}

\section{Background on looped graphs}

It is convenient to work with graphs with at most one loop permitted at each vertex. For such a looped graph $G$ on $n$ vertices, the \emph{adjacency matrix }is the symmetric $n\times n$ $\{0, 1\}$-matrix $A = (a_{ij})_{1\le i, j\le n}$ such that $a_{ij} = 1$ if $ij \in E(G)$ and $a_{ij} = 0$ otherwise. The \emph{$t$-blowup} of a looped graph $G$ is the graph $G(t)$ with vertex-set $V(G(t)) = \{ v_{i, j}: 1\le i\le n, 1\le j\le t\}$ and edge set $E(G(t)) = \{v_{i_1, j_1}v_{i_2,j_2}: i_1i_2 \in E(G)\}$. (Informally, this is the analogue for looped graphs of the ordinary $t$-blowup of a simple graph, where each vertex is replaced by an independent set of size $t$ and edges are replaced by complete bipartite graphs $K_{t, t}$, but some care needs to be taken to work with looped vertices).

The main use of $t$-blowup graphs in the study of extremal linear combinations of graph eigenvalues is that their eigenvalues are simply scaled versions of the underlying host graph (along with additional $0$s), so individual graphs can be used to obtain lower bounds on the corresponding limit problem. 

Let $\mathcal G_n^\circ$ denote the family of graphs on $n$ vertices with loops
allowed. For $G\in\mathcal G_n^\circ$, define the complement by
\[
A(G)+A(\overline G)=J_n,
\]
where $J_n$ is the $n\times n$ all-$1$s matrix. (Hereafter, the subscript $n$ will be omitted when it is clear what it is from context). 

Define
\[
\alpha_{i,j}^\circ
=
\lim_{n\to\infty}
\frac1n
\max_{G\in\mathcal G_n^\circ}
\bigl(
\lambda_i(G)+\lambda_j(\overline G)
\bigr)
\]
and
\[\beta_{i,j}^\circ = \lim_{n \to\infty}\frac{1}{n}\min_{G\in\mathcal{G}_n^\circ}(\lambda_{n-i+1}(G) + \lambda_{n-j+1}(\overline{G})).\]

Adapting an argument from \cite{BrooksLinzLu}, we show that adding loops to the graphs does not change the limit asymptotically. 

\begin{lemma}\label{lem:lim}
For every fixed $i$ and $j$,
\begin{align*}
\alpha_{i,j}&=\alpha_{i,j}^\circ\\
\beta_{i,j}&=\beta_{i,j}^\circ
\end{align*}
\end{lemma}

\begin{proof}
Adding or deleting loops changes the adjacency matrix by a
diagonal $0$--$1$ matrix $D$.
Hence every eigenvalue changes by at most
$\|D\|=1$ by Weyl's inequality.
Therefore the corresponding extremal quantities differ by at
most $2$.
After division by $n$ and letting $n\to\infty$,
the difference vanishes.
\end{proof}

\section{Maximizing \texorpdfstring{$\lambda_1(G) + \lambda_2(\overline{G})$}{}}

We now state and prove a lemma from which we will be able to quickly prove Theorems~\ref{thm:alpha11exact} and ~\ref{thm:alpha12}. 

\begin{lemma}\label{lem:lam1kupper}
	Let $G$ be a looped graph on $n$ vertices and let $1\leq k \leq n$. If there exist nonzero vectors $y_1, \dots, y_k \in \mathbb{R}_{\geq 0}^n $ with pairwise disjoint supports such that $A(\overline G) y_r \geq \lambda_k(\overline G) y_r $ for all $1\leq r \leq k$, then 
	\begin{equation}
		\lambda_1(G)+\lambda_k(\overline G)\leq \frac{4k}{4k-1}n.
	\end{equation}
\end{lemma}
\begin{proof}
	Let $A = A(G)$ and $B = A(\overline G)$. Define $\lambda = \lambda_1(A)$ and $\mu = \lambda_k(B)$. If $\mu \leq 0$ or $\lambda = 0$, then $\lambda + \mu \leq n < \frac{4k}{4k-1}n$. We assume $\lambda, \mu >0$. 
	
	Let $x \geq 0$ be a Perron eigenvector of $A$. Define $c = \sum_{i=1}^{n} x_i$. Then
	\[
	\lambda x = Ax = (J-B)x = c\mathbf{1} - Bx \leq c\mathbf{1}.
	\]
	
	We define the deficit vector $d = \frac{c}{\lambda}\mathbf{1} - x\geq 0 $. For a fixed $1\leq r \leq k$, let $p := y_r$ and define $P = \sum_{i=1}^{n}p_i$.	Since $Bp \geq \mu p $, we have
	\begin{align*}
		\lambda p^T x &= p^T Ax\\
		&=p^T(J-B)x \\
		&=cP -p^T B x\\
		&=cP - (Bp)^T x\\
		&\leq cP - \mu p^T x.
	\end{align*}
	Therefore, $p^T x \leq \frac{cP}{\lambda + \mu}$. We consider the quantity $p^Td$. On one hand, we have
	\[
	p^T d = \sum_{i=1}^{n}p_i \left(\frac{c}{\lambda} - x_i \right) = \frac{cP}{\lambda} - p^T x \geq \frac{cP}{\lambda} - \frac{cP}{\lambda + \mu} = \frac{cP\mu}{\lambda(\lambda + \mu)}.
	\]
	On the other hand, since
	\[
	\mu p_i \leq (Bp)_i = \sum_{j=1}^{n}B_{ij}p_j \leq \sum_{j=1}^{n}p_j = P,
	\]
	we have $p_i \leq P/\mu$ and as $d \geq 0$
	\[
	p^T d = \sum_{i=1}^{n}p_i d_i  = \sum_{\text{supp}(p)}p_i d_i \leq \frac{P}{\mu}\sum_{\text{supp}(p)}d_i.
	\]
	Hence, for any $y_r$ we have
	\[
	\sum_{\text{supp}(y_r)} d_i \geq \frac{c\mu^2}{\lambda(\lambda+\mu)}.
	\]
	
	Now, since the vectors $y_1, \dots, y_k$ have pairwise disjoint support, it follows that
	\[
	\sum_{i=1}^{n} d_i \geq \frac{kc\mu^2}{\lambda(\lambda+\mu)}.
	\]
	Furthermore
	\[
	\frac{nc}{\lambda} - c = \sum_{i=1}^{n} \left(\frac{c}{\lambda} - x_i \right) = \sum_{i=1}^{n} d_i \geq \frac{kc\mu^2}{\lambda(\lambda+\mu)}.
	\]
	Since $c, \lambda, \mu > 0 $ we have
	\[
	(n- \lambda)(\lambda + \mu) \geq k\mu^2.
	\]
	
	We observe
	\begin{align*}
		&&\quad(n- \lambda)(\lambda + \mu) &\geq k\mu^2\\
		&\Leftrightarrow&\quad(n+\mu - (\lambda+\mu))(\lambda + \mu) &\geq k\mu^2\\
		&\Leftrightarrow&\quad n(\lambda+\mu) - (\lambda + \mu)^2 &\geq k\mu^2 - \mu(\lambda+\mu)\\
		&\Leftrightarrow&\quad n(\lambda+\mu) - \frac{4k-1}{4k}(\lambda+\mu)^2&\geq k\left(\mu - \frac{\lambda+\mu}{2k}\right)^2\\
		&\Rightarrow&\quad n(\lambda+\mu) - \frac{4k-1}{4k}(\lambda+\mu)^2&\geq 0.
	\end{align*}
	Thus
	\[
	\lambda+\mu \leq \frac{4k}{4k-1}n.
	\]
\end{proof}

An immediate consequence of Lemma~\ref{lem:lam1kupper} is a short new proof of the fact that $\alpha_{1, 1} = \frac43$. 
\begin{theorem}\label{thm:alpha11}
	For any looped graph $G$ on $n$ vertices, we have
	\[
	\lambda_1(G) + \lambda_1(\overline G) \leq \frac{4}{3}n.
	\]
\end{theorem}
\begin{proof}
    Let $B = A(\overline G)$. By Perron-Frobenius, there exists a vector $p\geq 0$ such that $Bp = \mu p$. Hence, the hypotheses of Lemma~\ref{lem:lam1kupper} are satisfied, and we conclude $\lambda_1(G) + \lambda_1(\overline G) \leq \frac{4}{3}n.$
\end{proof}
\begin{remark}
    Lemma~\ref{lem:lam1kupper} is stated for looped graphs. If one walks through the same argument assuming that the graph $G$ is simple, then the obtained upper bound is
    \[\lambda_1(G) + \lambda_k(\overline{G}) \le \frac{4k}{4k-1}n - 1,\]
    under the hypotheses of Lemma~\ref{lem:lam1kupper}. 
    Setting $k=1$ recovers Terpai's theorem. Setting $k=2$ gives Theorem~\ref{thm:alpha12graphs} when combined with the argument in the proof of Theorem~\ref{thm:alpha12} given in the sequel. 
\end{remark}
We now use Lemma~\ref{lem:lam1kupper} to prove Theorem~\ref{thm:alpha12}, which states that
	\[
	\lambda_1(G) + \lambda_2(\overline G) \leq \frac{8}{7}n.
	\]
for any looped graph $G$ on $n$ vertices.
This shows that $\alpha_{1, 2} = \frac87$. 
\begin{proof}[Proof of Theorem~\ref{thm:alpha12}]
	Let $A = A(G)$ and $B = A(\overline G)$. Define $\lambda = \lambda_1(A)$ and $\mu = \lambda_2(B)$.
	
	Suppose $\overline G$ has two distinct connected components $C_1, C_2 \subseteq \overline G$, with corresponding principal submatrices $B_1, B_2$ such that $\lambda_1(B_1) \geq \mu$ and $\lambda_1(B_2) \geq \mu$. Let $p, q \geq 0$ be Perron eigenvectors of $B_1$ and $B_2$ respectively. Extend $p$ and $q$ by zero outside their respective components. Then $Bp = \lambda_1(B_1)p \geq \mu p$, $Bq = \lambda_1(B_2) q \geq \mu q$ and $\text{supp}(p) \cap \text{supp}(q) = \emptyset$.
	
	Otherwise, $\overline G$ has exactly one connected component $C_1 \subseteq \overline G$ with corresponding principal submatrix $B_1$ such that $\lambda_1(B_1) \geq \mu$. Since $\mu = \lambda_2(B)$ and the Perron eigenvalues of all other connected components of $\overline G$ are strictly less than $\mu$, we have that $\mu$ is an eigenvalue of $B_1$. We show that $\lambda_1(B_1) > \mu$. Suppose $\lambda_1(B_1) = \mu$. Then $B_1$ would only contribute one eigenvalue at least $\mu$ to the spectrum of $B$ as $C_1$ is connected and its Perron eigenvalue is simple. However, the Perron eigenvalues of all other components of $\overline G$ are strictly less than $\mu$, a contradiction.
	
	Let $x>0$ be a Perron eigenvector of $B_1$ and let $y$ be an eigenvector of $B_1$ associated with $\mu$. Since $\lambda_1(B_1) > \mu$ we have $y^Tx = 0$ which implies $y$ has both positive and negative entries. We decompose $y$ such that $y = y^+ - y^-$ where $y_i^+ = \max\{y_i, 0\}$ and $ y_i^- = \max\{-y_i, 0\}$. Let $p = y^+$, $q = y^-$ and extend $p, q$ by zero outside of $C_1$. We observe
	
	\[
		Bp - Bq = B(p - q) = \mu(p - q)  = \mu p - \mu q.
	\]
	
	If $p_i >0$, then $q_i = 0 $ and so $(Bp)_i - \mu p_i = (Bq)_i \geq 0$. If $p_i = 0$, then $(Bp)_i \geq 0 = \mu p_i$. Therefore, $Bp \geq \mu p$. A similar argument yields $Bq \geq \mu q$. Note that $\text{supp}(p) \cap \text{supp}(q) = \emptyset$. 
    
    In both cases, the hypotheses of Lemma~\ref{lem:lam1kupper} are satisfied, and we conclude $\lambda_1(G) + \lambda_2(\overline G) \leq \frac{8}{7}n.$
    \end{proof}

As mentioned in the introduction, we are unable to determine the maximum value of $\lambda_1(G) + \lambda_k(\overline{G})$. The difficulty is that for $k\ge 3$ we cannot show that an extremal graph must satisfy the hypotheses of Lemma~\ref{lem:lam1kupper}. 
We can, however, prove Theorem~\ref{thm:alpha1kub}, which states that for any looped graph $G$ on $n$ vertices and any $k\ge 2$, 
    \[\lambda_1(G) + \lambda_k(\overline{G}) \le \frac{k + \sqrt{k(4k-1)}}{3k-1}n. \]

The proof of Theorem~\ref{thm:alpha1kub} uses trace equalities and the Weyl  inequalities to obtain a quadratic inequality in the parameters $\lambda_1(G)$ and $\lambda_k(\overline{G})$ that any graph $G$ must satisfy. Lagrange multipliers are then used to obtain a good bound on $\lambda_1(G) + \lambda_k(\overline{G})$.

\begin{proof}[Proof of Theorem~\ref{thm:alpha1kub}]
    Let $G$ be a looped graph on $n$ vertices with eigenvalues $\lambda_1 \ge \ldots \ge \lambda_n$ which maximizes $\lambda_1(G) + \lambda_k(\overline{G})$. We may assume that $\lambda_1(G) + \lambda_k(\overline{G}) \ge \frac{4k}{4k-1}n$, so that in particular $\lambda_k(\overline{G}) > 0$.  Henceforth, we denote the eigenvalues of $\overline{G}$ by $\mu_1 \ge \mu_2 \ge \ldots \ge \mu_n$. 
    By taking the trace of both sides of the equation $A(G)+ A(\overline{G}) = J$, we find
    \begin{align}
    n &= (\lambda_1 + \ldots + \lambda_n) + (\mu_1 + \ldots + \mu_n) \nonumber\\
    & \ge (\lambda_1 + \lambda_{n-k+1} +\ldots  +\lambda_{n-1} + \lambda_n) + (\mu_1 + \mu_2 +\ldots + \mu_k + \mu_n).  \label{eqn:trAAbar}
    \end{align}
    In the second line, we used $\lambda_i + \mu_{n+1-i} \ge \lambda_n(J) = 0 $, which follows from Weyl's inequality. Let $z \ge 0$ be a real number. By Weyl's inequality, for all $i$ such that $n-k+2 \le i \le n$, we have that $\frac{z}{k-1}(\mu_{k} + \lambda_{i}) \le 0$. Adding these $k-1$ inequalities together gives the inequality 
    \begin{equation}\label{eqn:averaging}
        z\mu_k + \frac{z}{k-1}\lambda_{n-k+2} + \ldots + \frac{z}{k-1}\lambda_n \le 0. 
    \end{equation}
    Combining \eqref{eqn:trAAbar} and \eqref{eqn:averaging}, we obtain
    \begin{equation}\label{eqn:trAAbarz}
    n \ge (\lambda_1 + \lambda_{n-k+1} + \left(1 + \frac{z}{k-1}\right)(\lambda_{n-k+2}\ldots  +\lambda_{n-1} + \lambda_n)) + (\mu_1 + \mu_2 +\ldots + \mu_{k-1} + (1 + z)\mu_k + \mu_n)
    \end{equation}
    
    By computing the traces of $A(G)^2$ and $A(\overline{G})^2$, we find 
    \begin{align}\label{eqn:trAAbar2}
n^2 &= (\lambda_1^2 + \ldots + \lambda_n^2) + (\mu_1^2 + \ldots +\mu_n^2) \nonumber \\
&\ge (\lambda_1^2 +\lambda_{n-k+1}^2 + \ldots +\lambda_{n-1}^2 + \lambda_n^2) + (\mu_1^2 + \ldots + \mu_k^2 + \mu_n^2).
    \end{align}
    From \eqref{eqn:trAAbarz} and the inequalities $\mu_1 \ge \mu_2 \ge \ldots \ge \mu_k$, we obtain
    \begin{equation}\label{eqn3}
    -(\lambda_{n-k+1} + \mu_n + \left(1 + \frac{z}{k-1}\right)(\lambda_{n-k+2} + \ldots + \lambda_n)) \ge \lambda_1 + (k+z)\mu_k - n.
    \end{equation}
    By Cauchy-Schwarz and \eqref{eqn3}, we have 
    \begin{equation}\label{eqn4}
        \lambda_{n-k+1}^2 + \ldots  + \lambda_n^2 + \mu_n^2 \ge \frac{1}{2 + (k-1)\left(1 + \dfrac{z}{k-1}\right)^2}(\lambda_1 + (k+z)\mu_k-n)^2
    \end{equation}
    Combining \eqref{eqn:trAAbar2} and \eqref{eqn4} gives the inequality
    \begin{equation}\label{eqn:lambda1mu2}
    n^2 \ge \lambda_1^2 + k\mu_k^2 + \frac{1}{2 + (k-1)\left(1 + \dfrac{z}{k-1}\right)^2}(\lambda_1 + (k+z)\mu_k - n)^2.
    \end{equation}

The inequality \eqref{eqn:lambda1mu2} is valid for any $z \ge 0$. We now make the specific choice $z = (k-1)\left(1 + \sqrt{4 - \frac{1}{k}}\right)$. Inequality~\eqref{eqn:lambda1mu2} becomes
\begin{equation}\label{eqn:lambda1mu2z}
    n^2 \ge \lambda_1^2 + k\mu_k^2 + \frac{1}{2 + (k-1)\left(2 + \sqrt{4-\frac{1}{k}}\right)^2}\left(\lambda_1 + \left(k + (k-1)\left(1 + \sqrt{4-\frac{1}{k}}\right)\right)\mu_k-n\right)^2
\end{equation}
We now solve the following optimization problem:  maximize $\lambda_1 + \mu_k$ subject to inequality \eqref{eqn:lambda1mu2z}. The maximum will be achieved on the boundary of the ellipse, so we can model this as a Lagrange multipliers problem. 

Let \[f(\lambda_1, \mu_k) = \lambda_1 + \mu_k\] and \[g(\lambda_1, \mu_k) = \lambda_1^2 + k\mu_k^2 + \frac{1}{2 + (k-1)\left(2 + \sqrt{4-\frac{1}{k}}\right)^2}\left(\lambda_1 + \left(k + (k-1)\left(1 + \sqrt{4-\frac{1}{k}}\right)\right)\mu_k-n\right)^2 - n^2.\] To make the formulas more compact, we will set $a = 2 + \sqrt{4-\frac{1}{k}}$. The Lagrange multipliers problem is to solve the equation $\nabla f = \lambda \nabla g$ subject to the condition $g = 0$. 
The equation $\nabla f = \lambda \nabla g$ gives us two equations:
\begin{align*}
1& = \lambda(2\lambda_1 + \frac{2}{2 + (k-1)a^2}(\lambda_1 + (1 + a(k-1))\mu_k - n))  \\
1&= \lambda(2k\mu_k + \frac{2(1 + a(k-1))}{2+(k-1)a^2}(\lambda_1 + (1+a(k-1))\mu_k -n)).
\end{align*}
Equating the right-hand sides of these two equations, we find
\[ 2\lambda_1 + \frac{2}{2 + (k-1)a^2}(\lambda_1 + (1 + a(k-1))\mu_k - n) = 2k\mu_k + \frac{2(1 + a(k-1))}{2+(k-1)a^2}(\lambda_1 + (1+a(k-1))\mu_k -n) \]
\[ \implies \lambda_1 = \frac{\mu_k(a^2(2k^2-3k+1) + a(k-1) + 2k) - a(k-1)n}{a^2(k-1)-ak+a+2}. \]
Substituting this expression for $\lambda_1$ into the condition $g(\lambda_1, \mu_k) = 0$ gives
a quadratic equation for $\mu_k$. Solving the quadratic, we find the solution of $\mu_k$ corresponding to the maximum is 
\[\mu_k = \frac{na}{2(ka-1)}. \]
Solving for $\lambda_1$, we find
\[\lambda_1 = \frac{(2k-1)na}{2(ka-1)}\]
Using the previous relation between $\lambda_1$ and $\mu_k$, it follows that 
\begin{align*}
    \lambda_1 + \mu_k & \le \frac{2kna}{2(ka-1)} \\
    &=  \frac{2k\left(2 + \sqrt{4 - \frac{1}{k}}\right)}{2\left(2k + k\sqrt{4 - \frac{1}{k}} - 1\right)} n \\
    &=\frac{2k + \sqrt{k(4k-1)}}{2k + \sqrt{k(4k-1)}-1} n\\
    & = \frac{k + \sqrt{k(4k-1)}}{3k-1} n.
\end{align*}
\end{proof}

\section{General bounds on \texorpdfstring{$\alpha_{i, j}$}{}}

For a looped graph $G$, since the eigenvalues of $J$ are
\[
n,0,\ldots,0,
\]
we obtain the following simple relation between the eigenvalues of $G$ and $\overline{G}$ when $G$ is regular. 

\begin{lemma}\label{lem:reg}
    Let $G$ be a regular looped graph on $n$ vertices. The eigenvalues of $\overline{G}$ are 
    \[n - \lambda_1(G), -\lambda_2(G), \ldots, -\lambda_n(G). \]
    In particular, if $j\ge 2$ and $\lambda_j(G) \ge 0$, then
    \[ \lambda_j(G) = - \lambda_{n-j+2}(\overline{G}). \]
\end{lemma}

We now use Weyl's inequalities to obtain various constraints on
\[
\lambda_i(G)+\lambda_j(\overline G).
\]
Our first goal is to establish general bounds on
$\alpha_{i, j}$ using the spectral gaps defined by Brooks, Linz and Lu.

\begin{lemma}\label{lem:s_ub}
For every integer $i\ge 1$ and $j\geq 2$,
\[
s_{i-1, j-1} \le \alpha_{i,j} \le s_{i-1,j-2}. 
\]
Furthermore, if there is a sequence $G_n$ of regular graphs such that \[\lim_{n\rightarrow \infty}\frac{\lambda_i(G_n) - \lambda_{n-j+2}(G_n)}{n} = s_{i-1, j-2},\] then $\alpha_{i, j} = s_{i-1, j-2}$. 
\end{lemma}

\begin{proof}
For any graph $G$ on $n$ vertices, we have
\[
-\lambda_{n-j+1}(G) \le \lambda_j(\overline G)
\le
-\lambda_{n-j+2}(G),
\]
which follows from the complement relation
\[
A(G)+A(\overline G)=J
\]
and Weyl's inequalities. Consequently, 
\begin{equation}\label{eqn:alphaublb}
\lambda_i(G) - \lambda_{n-j+1}(G) 
\le
\lambda_i(G)+\lambda_j(\overline G)
\le
\lambda_i(G)-\lambda_{n-j+2}(G).
\end{equation}

Let $(G_n)$ be a sequence of graphs attaining asymptotically the
maximum defining $s_{i-1,j-2}$. By the definition of
$s_{i-1,j-2}$,
\[
s_{i-1,j-2}
=
\lim_{n\to\infty}\frac1n
\bigl(
\lambda_i(G_n)-\lambda_{n-j+2}(G_n)
\bigr).
\]

Inequality \eqref{eqn:alphaublb} applied to the graph sequence $(G_n)$ yields
\[
\alpha_{i,j}\le s_{i-1,j-2}
\]
after dividing by $n$ and letting $n\rightarrow \infty$. We obtain $\alpha_{i, j} \ge s_{i-1, j-1}$ similarly.

If $G_n$ is regular and $n$ is large enough, then by Lemma~\ref{lem:reg} we obtain equality in our previous application of Weyl's inequalities. In particular, if the graphs $G_n$ in the sequence achieving the maximum defining $s_{i-1, j-2}$ are regular, then the sequence of graphs $G_n$ also asymptotically achieves $\alpha_{i, j}$.

\end{proof}


We recall the general upper bound for $s_{i,j}$ proved by Brooks, Linz and Lu~\cite{BrooksLinzLu}. The bound is not tight in general, but it is tight for infinitely many $k$ when $i=k$ and $j=k-1$. 
\begin{theorem}[Brooks--Linz--Lu~\cite{BrooksLinzLu}]
\label{thm:genijub}
    For any positive integer $i$ and nonnegative integer $j$ 
    \[s_{i,j} \le \frac{1}{2}\sqrt{\frac{i+j+1}{i(j+1)}}.\]
    If a symmetric Hadamard matrix of order $2k$ exists, then 
    \[s_{k, k-1} = \frac{1}{\sqrt{2k}}.\]
\end{theorem}

The matching lower bound construction when $i=k$ and $j=k-1$ is as follows. Let $K = \begin{bmatrix} 1 & - 1 \\ - 1 & 1\end{bmatrix}$ and let $H$ be a symmetric Hadamard matrix of order $2k$. Then the looped graph $G$ on $4k$ vertices whose adjacency matrix is
\[\frac12(K\otimes H + J_{4k})\]
has $\lambda_{k+1}(G) - \lambda_{n-k+1}(G) = \frac{4k}{\sqrt{2k}}$. The $t$-blowups of $G$ give an infinite sequence of graphs on $n=4kt$ vertices with $\lambda_{k+1}(G) - \lambda_{n-k+1}(G) = \frac{n}{\sqrt{2k}}$, proving tightness of the upper bound when the relevant Hadamard matrix exists.

Lemma~\ref{lem:s_ub} and  Theorem~\ref{thm:genijub} imply the general upper bounds for $\alpha_{i, j}$ stated in Theorem~\ref{thm:genijubng}.

\section{General bounds on \texorpdfstring{$\beta_{i, j}$}{}}

We first use Weyl's inequalities to relate the value of $\beta_{i, j}$ to the $(j-1, i-1)$-spread and the $(j, i-1)$-spread.

\begin{lemma}\label{lem:betaweyl}
    For any fixed $i,j\geq 1$, we have
    \[-s_{j-1,i-1} \leq \beta_{i,j}\leq -s_{j,i-1}.\]
\end{lemma}
\begin{proof}
    By Weyl's inequalities, we have
    \[ -\lambda_j(G)\leq \lambda_{n+1-j}(\overline G) \leq -\lambda_{j+1(G)}. \]
    This implies
    \[ -(\lambda_j(G)-\lambda_{n+1-i}(G))\leq \lambda_{n+1-i}(G) + \lambda_{n+1-j}(\overline G) \leq -(\lambda_{j+1}(G)-\lambda_{n+1-i}(G)) . \]
\end{proof}

We now prove Theorem~\ref{thm:betaij}, which states that  \[|\beta_{i, j}| \le \frac12 \sqrt{\frac{1}{i} + \frac{1}{j}}.\]
Note that Theorem~\ref{thm:betaij} is a strengthening of Theorem~\ref{thm:genijub}. Indeed, by Lemma~\ref{lem:betaweyl}, we obtain \[s_{j, i-1}\le |\beta_{i, j}| \le \frac12 \sqrt{\frac{1}{i} + \frac{1}{j}},\] which recovers Theorem~\ref{thm:genijub}. Corollary~\ref{cor:betakk} in the introduction is obtained by setting $i=j=k$ in Theorem~\ref{thm:betaij} and using the same Hadamard-matrix based construction as for the $(k, k-1)$-spread described in the discussion after Theorem~\ref{thm:genijub}. Thus, the exact value of $\beta_{k, k}$ is determined for infinitely many $k$. 

We now give the proof of Theorem~\ref{thm:betaij}. The proof is similar to the proof of Theorem~\ref{thm:alpha1kub}.

\begin{proof}[Proof of Theorem~\ref{thm:betaij}]
Let the eigenvalues of $G$ be denoted by $\lambda_1 \ge \ldots \ge \lambda_n$ and the eigenvalues of $\overline{G}$ be denoted by $\mu_1 \ge \ldots \ge \mu_n$. We may assume that $\lambda_{n-i+1} < 0$ and $\mu_{n-j+1} < 0$. Indeed, if $\lambda_{n-i+1} \ge 0$, then we have that \[\lambda_{n-i+1} + \mu_{n-j+1} \ge \mu_{n-j+1} \ge -\frac{n}{2\sqrt{j}} - 1 \ge -\frac{n}{2} \sqrt{\frac{1}{i} + \frac{1}{j}} - 1,\]
where the second inequality above follows from a result of Nikiforov~\cite[Theorem 2.6]{Nikiforov2015}. Dividing through by $n$ and taking limits gives the desired inequality. A similar argument shows that we may assume $\mu_{n-j+1} < 0$. 

Using that $n = \tr(A(G)) + \tr(A(\overline{G}))$, we have
\[ n = (\lambda_1 + \ldots + \lambda_n) + (\mu_1 + \ldots + \mu_n). \]
By Weyl's inequality, we have $\lambda_{k+2} + \mu_{n-k} \le \lambda_2(J) = 0$ for $0\le k\le n-2$. Let $z$ and $w$ be real numbers with $0\le z \le 1$ and $0\le w\le 1$. It follows from the trace equality and Weyl's inequality that
\[
 n \le (\lambda_1 + w\lambda_2 + \ldots + w\lambda_{j+1}) + (z\lambda_{n-i+1} + \ldots + z\lambda_n) + (\mu_1 + z\mu_2 + \ldots + z\mu_{i+1}) + (w\mu_{n-j+1} + \ldots + w\mu_n).
\]
This implies
\begin{equation}\label{eqn:betaijtr}
\lambda_1 + \mu_1 + w(\lambda_2 + \ldots + \lambda_{j+1}) + z(\mu_2 + \ldots + \mu_{i+1}) \ge n - z(\lambda_{n-i+1} + \ldots + \lambda_{n}) -w(\mu_{n-j+1} + \ldots + \mu_n).
\end{equation}
Now we use the identity $n^2 = \tr(A^2(G)) + \tr(A^2(\overline{G}))$ to obtain the inequality
\[
    n^2 \ge \lambda_1^2 + \lambda_2^2 + \ldots + \lambda_{j+1}^2 + \lambda_{n-i+1}^2 + \ldots + \lambda_n^2 + \mu_1^2 + \ldots + \mu_{i+1}^2 + \mu_{n-j+1}^2 + \ldots + \mu_n^2.
\]
By Cauchy-Schwarz and \eqref{eqn:betaijtr}, we obtain

\begin{equation}\label{eqn:n^2ineq}
n^2 \ge \frac{(n-zi\lambda_{n-i+1}-wj\mu_{n-j+1})^2}{2 + jw^2+iz^2}+i\lambda_{n-i+1}^2 + j\mu_{n-j+1}^2.
\end{equation}
We divide through by $n^2$ and set $x:=-\frac{\lambda_{n-i+1}}{n}, y:=-\frac{\mu_{n-j+1}}{n}$. Then \eqref{eqn:n^2ineq} becomes
\begin{equation}\label{eqn:wzineq}
1 \ge \frac{(1+zix+wjy)^2}{2 + jw^2+iz^2}+ix^2 + jy^2.
\end{equation}
We maximize $x+y$ subject to inequality \eqref{eqn:wzineq}. For fixed $w$ and $z$, the maximum of $x + y$ will be achieved on the boundary of the ellipse, so we can treat \eqref{eqn:wzineq} as an equality. We now make the specific choices $z = \sqrt{\frac{j}{i(i+j)}}$ and $w = \sqrt{\frac{i}{j(i+j)}}$. Then note 
\begin{equation}\label{eqn:zwij}
jw^2 + iz^2 = 1 \text{ and } zi = wj = \sqrt{\frac{ij}{i+j}}.
\end{equation}Let $f(x, y) = x + y$ and $g(x, y)= \frac{\left(1 + \sqrt{\frac{ij}{i+j}}(x+y)\right)^2}{3} + ix^2+jy^2 -1$, so that the optimization problem is now: maximize $f(x, y)$ subject to the condition $g(x, y) = 0$. Using Lagrange multipliers, from the gradient equation $\nabla f = \lambda \nabla g$, we obtain the two equations 
\[1 = \lambda\left(\frac{2iz(1+izx + jwy)}{2+jw^2 + iz^2} + 2ix\right)\]
\[1 = \lambda\left(\frac{2jw(1+izx + jwy)}{2+jw^2 + iz^2} + 2jy\right)\]
Setting the right-hand sides of these equations equal to each other and using the relations in \eqref{eqn:zwij}, we obtain that
\[ix = jy.\]
Substituting $y = \frac{ix}{j}$ into the equation $g(x, y) = 0$ gives
\[x = \frac12\sqrt{\frac{j}{i(i+j)}}.\]
Hence,
\[ y = \frac{i}{j} x = \frac12\sqrt{\frac{i}{j(i+j)}}. \]
Thus, we obtain the upper bound
\[x + y \le \frac12\sqrt{\frac{j}{i(i+j)}} + \frac12\sqrt{\frac{i}{j(i+j)}} = \frac12\sqrt{\frac{1}{i} + \frac{1}{j}}, \]
completing the proof. 
\end{proof}

\section{Some lower bound constructions for \texorpdfstring{$\alpha_{i, j}$}{} and \texorpdfstring{$\beta_{i, j}$}{}}

It would be interesting to determine $\alpha_{i, j}$, $\beta_{i, j}$ and $s_{i, j}$ for more pairs $(i, j)$. For the $(i, j)$-spread problem, Brooks, Linz and Lu~\cite{BrooksLinzLu} gave tables with the then-best known lower bounds for $s_{i, j}$. We update the table of best-known lower bounds for $s_{i, j}$ and create similar tables for $\alpha_{i, j}$ and $\beta_{i, j}$. The complete tables are given in an appendix. Here, we mention a few graphs which improve over the previous best-known constructions for $s_{2, 3}$ and $s_{2, 2}$. 

\begin{proposition}
    \[\alpha_{3, 4}, s_{2, 2} \ge 0.4467; \quad \alpha_{3, 5} \ge 0.4219; \quad s_{2, 3} \ge 0.4278.  \]
\end{proposition}

 The current lower bounds of $\alpha_{3,4}$ and $s_{2,2}$ are achieved by $t$-blowups of the graph $G_1$ shown in Figure~\ref{fig:G1}.  The current lower bounds of $\alpha_{3,5}$ and $s_{2,3}$ are achieved by $t$-blowups of the graph $G_2$ whose complement $\overline{G_2}$ is shown in Figure~\ref{fig:G2bar}. Notice that the looped vertices induce a looped version of $Q_3$, the cube graph.

For both $G_1$ and $G_2$, all degrees are as close as possible to half the number of vertices. $G_1$ has 15 vertices and each of them has degree either 7 or 8. Similarly, $G_2$ has 13 vertices and each of them has degree either 6 or 7. From the matrix perspective, this means that each row of the adjacency matrices has as close as possible to the same number of ones and zeros.

\begin{figure}[h!]
    \centering
    \begin{tikzpicture}[scale=2, vertex/.style={scale=1, circle, draw=black, fill=black},
	wvertex/.style={scale=1, circle, draw=black, fill=white}]

  \node[vertex] (a1) at (120:2) {};
  \node[vertex] (a3) at (105:2){};
  \node[vertex] (a2) at (75:2) {};
  \node[wvertex] (a4) at (90:2) {};
  \node[wvertex] (a5) at (60:2) {};

  \node[vertex] (b1) at (0:2) {};
  \node[vertex] (b3) at (-15:2) {};
  \node[vertex] (b2) at (-45:2) {};
  \node[wvertex] (b4) at (-30:2) {};
  \node[wvertex] (b5) at (-60:2) {};

  \node[vertex] (c1) at (-120:2) {};
  \node[vertex] (c3) at (-135:2) {};
  \node[vertex] (c2) at (-165:2) {};
  \node[wvertex] (c4) at (-150:2) {};
  \node[wvertex] (c5) at (-180:2) {};

\foreach \u/\v in {
  a1/b2,a1/b4,a1/b5,
  a2/b1,a2/b3,
  a3/b2,a3/b4,a3/b5,
  a4/b1,a4/b3,a4/b5,
  a5/b2,a5/b4,a5/b5} \draw (\u) -- (\v);

\foreach \u/\v in {
  b1/c2,b1/c4,b1/c5,
  b2/c1,b2/c3,
  b3/c2,b3/c4,b3/c5,
  b4/c1,b4/c3,b4/c5,
  b5/c2,b5/c4,b5/c5} \draw (\u) -- (\v);

\foreach \u/\v in {
  c1/a2,c1/a4,c1/a5,
  c2/a1,c2/a3,
  c3/a2,c3/a4,c3/a5,
  c4/a1,c4/a3,c4/a5,
  c5/a2,c5/a4,c5/a5} \draw (\u) -- (\v);

  \foreach \x in {a,b,c} {
    \draw (\x1) -- (\x3);
    \draw (\x2) -- (\x4);
    \draw (\x2) -- (\x5);
  }

\end{tikzpicture}
    \caption{$G_1$, where the vertices with self-loops are colored in black.}
    \label{fig:G1}
\end{figure}
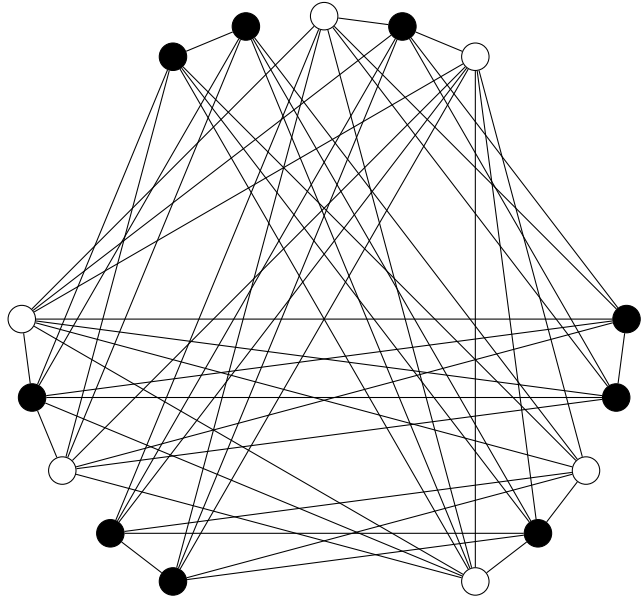

\begin{figure}[h!]
    \centering
    \begin{tikzpicture}[
    scale=3,
    vertex/.style={scale=1,circle,draw=black,fill=black},
    wvertex/.style={scale=1,circle,draw=black,fill=white}
]


\node[vertex] (5)  at (0,0) {};
\node[vertex] (7)  at (1,0) {};
\node[vertex] (11)  at (0,1) {};
\node[vertex] (8)  at (1,1) {};

\node[vertex] (9)  at (0.35,0.35) {};
\node[vertex] (6) at (1.35,0.35) {};
\node[vertex] (10) at (0.35,1.35) {};
\node[vertex] (12) at (1.35,1.35) {};


\node[wvertex] (1) at (2.70,1.55) {};
\node[wvertex] (0) at (2.45,1.05) {};
\node[wvertex] (4) at (2.45,0.25) {};
\node[wvertex] (3) at (2.70,-0.25) {};
\node[wvertex] (2) at (2.25,0.65) {};

\draw (0)--(1);
\draw (0)--(2);
\draw (0)--(3);
\draw (0)--(5);
\draw (0)--(6);
\draw (0)--(11);
\draw (0)--(12);

\draw (1)--(3);
\draw (1)--(4);
\draw (1)--(7);
\draw (1)--(9);
\draw (1)--(11);
\draw (1)--(12);

\draw (2)--(4);
\draw (2)--(7);
\draw (2)--(8);
\draw (2)--(9);
\draw (2)--(10);

\draw (3)--(4);
\draw (3)--(5);
\draw (3)--(6);
\draw (3)--(8);
\draw (3)--(10);

\draw (4)--(5);
\draw (4)--(6);
\draw (4)--(11);
\draw (4)--(12);

\draw (5)--(7);
\draw (5)--(9);
\draw (5)--(11);

\draw (6)--(7);
\draw (6)--(9);
\draw (6)--(12);

\draw (7)--(8);

\draw (8)--(11);
\draw (8)--(12);

\draw (9)--(10);

\draw (10)--(11);
\draw (10)--(12);

\end{tikzpicture}
    \caption{$\overline{G_2}$, where the vertices with self-loops are colored in black.}
    \label{fig:G2bar}
\end{figure}
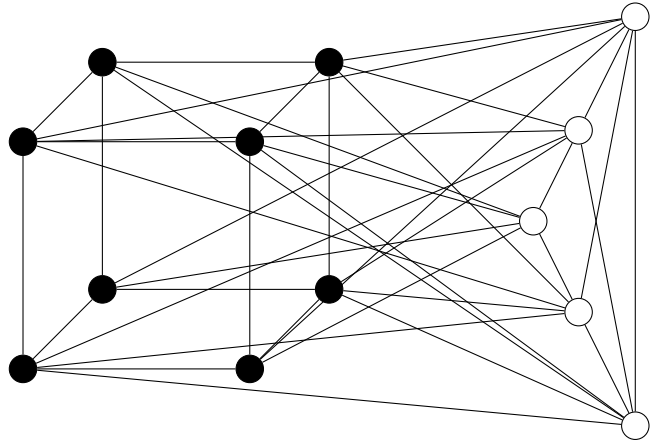

\section*{Acknowledgement}
ChatGPT 5.5 was used to help with several of the computations. ChatGPT 5.5 Pro also generated the proof idea for a special case of Lemma~\ref{lem:lam1kupper}. All of the writing was done by the authors, and we take full responsibility for the contents of the paper.

\bibliographystyle{plain}
\bibliography{NDreferences}

\newpage

\section{Appendix}

\subsection{Tables for \texorpdfstring{$\alpha_{i, j}$}{}}
\setlength{\tabcolsep}{2pt}
Here are the best known bounds for $\alpha_{i, j}$. 
\begin{table}[h!]
    \centering
    \begin{tabular}{|C{0.8cm}||*{5}{C{3cm}|}}
        \hline
        \diagbox[innerleftsep=4pt, innerwidth = 20pt]{$i$}{$j$} & 1 & 2 & 3 & 4 & 5 \\
        \hline\hline
        1 & $4/3$ & 8/7 & \mt{1.0909}{1.0931} & \mt{1.0667}{1.0678} & \mt{1.0526}{1.0533}\\
        \hline
        2 & 8/7 & $1/\sqrt{2}$  & \mt{0.6000}{0.6124} & \mt{0.5714}{0.5774} & \mt{0.5556}{0.5590}\\
        \hline
        3 & \mt{1.0909}{1.0931} & \mt{0.6000}{0.6124} & $1/2$  & \mt{0.4467}{0.4564} & \mt{0.4219}{0.4330}\\
        \hline
        4 & \mt{1.0667}{1.0678} & \mt{0.5714}{0.5774} & \mt{0.4467}{0.4564} & \mt{0.4041}{0.4082}& \mt{0.3678}{0.3819}\\
        \hline
        5 & \mt{1.0526}{1.0533} & \mt{0.5556}{0.5590} & \mt{0.4219}{0.4330} & \mt{0.3678}{0.3819}  & $\sqrt{2}/4$\\
        \hline
    \end{tabular}
    \caption{Known lower and upper bounds for $\alpha_{i,j}$.}
    \label{tab:boundsalpha}
\end{table}

\begin{table}[h!]
    \centering
    \begin{tabular}{|C{0.8cm}||*{5}{C{3cm}|}}
        \hline
        \diagbox[innerleftsep=4pt, innerwidth = 20pt]{$i$}{$j$} & 1 & 2 & 3 & 4 & 5 \\
        \hline\hline\xrowht{0.75cm}
        1\,\, & $J_1 \vee \overline{J_2}$ & $J_3 \vee \overline{2J_2}$ & $J_5 \vee \overline{3J_2}$  & $J_7 \vee \overline{4J_2} $ & $J_9 \vee \overline{5J_2}$  \\
        \hline\xrowht{0.75cm}
        2\,\, & $(J_3 \vee \overline{2J_2})^c$   & $P^*_4$~\cite{BrooksLinzLu} & $K_{3} \cup J_{2}$ & $K_{4} \cup J_{3}$ & $K_{5} \cup J_{4}$\\
        \hline\xrowht{0.75cm}
        3\,\, & $(J_5 \vee \overline{3J_2})^c$& $(K_{3} \cup J_{2})^c$ & $Q_3^*$~\cite{BrooksLinzLu} & $G_1$ & $G_2$ \\
        \hline\xrowht{0.75cm}
        4\,\, & $(J_7 \vee \overline{4J_2})^c$ & $(K_{4} \cup J_{3})^c$& $G_1^c$ & $G_4$ & $G_5$ \\
        \hline\xrowht{0.75cm}
        5\,\, & $(J_9 \vee \overline{5J_2})^c$  & $(K_{5} \cup J_{4})^c$ & $G_2^c$ & $G_5^c$ & $\frac{1}{2}(K \otimes H_8 + J_{16})$~\cite{BrooksLinzLu}\\
        \hline
    \end{tabular}
    \caption{Extremal graphs for $\alpha_{i,j}$.}
    \label{tab:graphsalpha}
\end{table}
\setlength{\tabcolsep}{6pt}

\begin{table}[h!]
    \centering
    \begin{tabular}{|c|c|}
        \hline
        Graph & sparse6 \\
        \hline\hline
        $G_1$ & \verb|:N@?KhBhOk`cJ?HAoAG\AfPGXbh_CqCi_COsIUq?QDMdZXBGSqTM|\\
        $G_2$ & \verb|:L@?KhQFQXYeN?HRF_COqOAGXCH_ChbgRGXAeMr|\\
        $G_4$ & \verb|:FehIA_t_S|\\
        $G_5$ & \verb|:S__@a`BaB`C_D_EFbCDEFG_@AHIaCEFGHJ`BDFGHJK_BCDEHJKL`ABCD|\\ &\verb|GILM`ABCEFILMN_ABDEFIKMO_@CDEGIJNP_ACDFGIKOQ_@BEFGIJPQR|\\
        \hline
    \end{tabular}
    \caption{Sparse6 representation for extremal graphs in Table \ref{tab:graphsalpha}.}
    \label{tab:sparse6alpha}
\end{table}

\newpage

\subsection{Tables for \texorpdfstring{$s_{i, j}$}{}}

Here are updated tables for the best known values of the spectral gap $s_{i, j}$.

\setlength{\tabcolsep}{2pt}

\begin{table}[h!]
    \centering
    \begin{tabular}{|C{0.8cm}||*{5}{C{3cm}|}}
        \hline
        \diagbox[innerleftsep=4pt, innerwidth = 20pt]{$i$}{$j$} & 0 & 1 & 2 & 3 & 4 \\
        \hline\hline
        0 & $2/\sqrt{3}$~\cite{BRTU2021+} & \mt{1.0902}{1.1124} & \mt{1.0664}{1.0774} & \mt{1.0524}{1.0590} & \mt{1.0434}{1.0477}\\
        \hline
        1 & $1/\sqrt{2}$ & \mt{0.6000}{0.6124} & \mt{0.5714}{0.5774} & \mt{0.5556}{0.5590} & \mt{0.5455}{0.5477}\\
        \hline
        2 & \mt{0.6000}{0.6123} & $1/2$ & \mt{0.4467}{0.4564} & \mt{0.4278}{0.4330} & \mt{0.4042}{0.4183}\\
        \hline
        3 & \mt{0.5714}{0.5774}  & \mt{0.4467}{0.4564} & \mt{0.4041}{0.4082} & \mt{0.3682}{0.3819} & \mt{0.3406}{0.3651}\\
        \hline
        4 & \mt{0.5556}{0.5590} & \mt{0.4219}{0.4330} & \mt{0.3678}{0.3819} & $\sqrt{2}/4$ & \mt{0.3149}{0.3354}\\
        \hline
    \end{tabular}
    \caption{Known lower and upper bounds for $s_{i,j}$.}
    \label{tab:bounds}
\end{table}

\begin{table}[h!]
    \centering
    \begin{tabular}{|C{0.8cm}||*{5}{C{3cm}|}}
        \hline
        \diagbox[innerleftsep=4pt, innerwidth = 20pt]{$i$}{$j$} & 0 & 1 & 2 & 3 & 4 \\
        \hline\hline\xrowht{0.75cm}
        0\,\, & $K_{3}^{(2)*}$ \cite{BRTU2021+}  & $K_{6}^{(3)*}$ & $K_{8}^{(4)*}$ & $K_{10}^{(5)*}$ & $K_{12}^{(6)*}$\\
        \hline\xrowht{0.75cm}
        1\,\, & $P^*_4$~\cite{BrooksLinzLu} & $K_{3} \cup J_{2}$ & $K_{4} \cup J_{3}$ & $K_{5} \cup J_{4}$ & $K_{6} \cup J_{5}$\\
        \hline\xrowht{0.75cm}
        2\,\, & $(K_{3} \cup J_{2})^c$ & $Q_3^*$~\cite{BrooksLinzLu} & $G_1$ & $G_2$ & $G_3$ \\
        \hline\xrowht{0.75cm}
        3\,\, & $(K_{4} \cup J_{3})^c$ & $G_1^c$ & $G_4$ & $G_5$ & $G_6$ \\
        \hline\xrowht{0.75cm}
        4\,\, & $(K_{5} \cup J_{4})^c$ & $G_2^c$ & $G_5^c$ & $\frac{1}{2}(K \otimes H_8 + J_{16})$~\cite{BrooksLinzLu} & $G_7$ \\
        \hline
    \end{tabular}
    \caption{Extremal graphs for $s_{i,j}$.}
    \label{tab:graphs}
\end{table}
\setlength{\tabcolsep}{6pt}

\begin{table}[h!]
    \centering
    \begin{tabular}{|c|c|}
        \hline
        Graph & sparse6 \\
        \hline\hline
        $G_1$ & \verb|:N@?KhBhOk`cJ?HAoAG\AfPGXbh_CqCi_COsIUq?QDMdZXBGSqTM|\\
        $G_2$ & \verb|:L@?KhQFQXYeN?HRF_COqOAGXCH_ChbgRGXAeMr|\\
        $G_3$ & \verb|:Oc?GgbaMGqOL?PbsIWyIDK\AXcIXATOAGXW@CKawAK\ATk_CXAiUq?PEMlbV^|\\
        $G_4$ & \verb|:FehIA_t_S|\\
        $G_5$ & \verb|:S__@a`BaB`C_D_EFbCDEFG_@AHIaCEFGHJ`BDFGHJK_BCDEHJKL`ABCD|\\ &\verb|GILM`ABCEFILMN_ABDEFIKMO_@CDEGIJNP_ACDFGIKOQ_@BEFGIJPQR|\\
        $G_6$ & \verb|:K@GKPT?QXAecOhxBGWyG@CLC?bGSqTOAG`RhV|\\
        $G_7$ & \verb|:J`?S@oBG[aDeOpwbJCPsHaOhc^|\\
        \hline
    \end{tabular}
    \caption{Sparse6 representation for extremal graphs in Table \ref{tab:graphs}.}
    \label{tab:sparse6s}
\end{table}
\newpage 

\subsection{Tables for \texorpdfstring{$\beta_{i, j}$}{}}

Here are the best known bounds for $\beta_{i, j}$. 
\setlength{\tabcolsep}{2pt}
\begin{table}[h!]
    \centering
    \begin{tabular}{|C{0.8cm}||*{5}{C{3cm}|}}
        \hline
        \diagbox[innerleftsep=4pt, innerwidth = 20pt]{$i$}{$j$} & 1 & 2 & 3 & 4 & 5 \\
        \hline\hline
       1&$-1/\sqrt{2}$&\mt{-0.6124}{-0.6000}&\mt{-0.5774}{-0.5714}&\mt{-0.5590}{-0.5556}&\mt{-0.5477}{-0.5455}\\
        \hline
        2&\mt{-0.6124}{-0.6000}&$-1/2$&\mt{-0.4564}{-0.4467}&\mt{-0.4330}{-0.4278}&\mt{-0.4183}{-0.4049}\\
        \hline
        3&\mt{-0.5774}{-0.5714}&\mt{-0.4564}{-0.4467}&\mt{-0.4082}{-0.4041}&\mt{-0.3819}{-0.3683}&\mt{-0.3651}{-0.3436}\\
        \hline
       4&\mt{-0.5590}{-0.5556}&\mt{-0.4330}{-0.4278}&\mt{-0.3819}{-0.3683}&$-\sqrt{2}/4$&\mt{-0.3354}{-0.3149}\\
        \hline
        5&\mt{-0.5477}{-0.5455}&\mt{-0.4183}{-0.4049}&\mt{-0.3651}{-0.3436}&\mt{-0.3354}{-0.3149}&\mt{-0.3162}{-0.3149}\\
        \hline
    \end{tabular}
    \caption{Known lower and upper bounds for $\beta_{i,j}$.}
    \label{tab:boundsbeta}
\end{table}
\begin{table}[h!]
    \centering
    \begin{tabular}{|C{0.8cm}||*{5}{C{3cm}|}}
        \hline
        \diagbox[innerleftsep=4pt, innerwidth = 20pt]{$i$}{$j$} & 1 & 2 & 3 & 4 & 5 \\
        \hline\hline\xrowht{0.75cm}
        1\,\, & $P^*_4$~\cite{BrooksLinzLu} & $(K_{3} \cup J_{2})^c$ & $(K_{4} \cup J_{3})^c$ & $(K_{5} \cup J_{4})^c$ & $(K_{6} \cup J_{5})^c$\\
        \hline\xrowht{0.75cm}
        2\,\, & $K_{3} \cup J_{2}$ & $Q_3^*$~\cite{BrooksLinzLu} & $G_1^c$ & $G_2^c$ & $G_3^c$\\
        \hline\xrowht{0.75cm}
        3\,\, & $K_{4} \cup J_{3}$ & $G_1$ & $G_4$ & $G_5^c$ & $G_6^c$ \\
        \hline\xrowht{0.75cm}
        4\,\, & $K_{5} \cup J_{4}$ & $G_2$ & $G_5$ & $\frac{1}{2}(K \otimes H_8 + J_{16})$~\cite{BrooksLinzLu} & $G_7^c$ \\
        \hline\xrowht{0.75cm}
        5\,\, & $K_{6} \cup J_{5}$ & $G_3$ &  $G_6$ & $G_7$ & $G_8$ \\
        \hline
    \end{tabular}
    \caption{Extremal graphs for $\beta_{i,j}$.}
    \label{tab:graphsbeta}
\end{table}
\begin{table}[h!]
    \centering
    \begin{tabular}{|c|c|}
        \hline
        Graph & sparse6 \\
        \hline\hline
        $G_1$ & \verb|:N@?KhBhOk`cJ?HAoAG\AfPGXbh_CqCi_COsIUq?QDMdZXBGSqTM|\\
        $G_2$ & \verb|:L@?KhQFQXYeN?HRF_COqOAGXCH_ChbgRGXAeMr|\\
        $G_3$ & \verb|:P?aAaB_@AB__DE_DEF`DEFG`DEFG`DGHI`DGHIJ`|\\
        & \verb|ABHIJK_ABCEF_ABCDEFL`ABCHIJKL`ABCIJKLMN|\\
        $G_4$ & \verb|:FehIA_t_S|\\
        $G_5$ & \verb|:S__@a`BaB`C_D_EFbCDEFG_@AHIaCEFGHJ`BDFGHJK_BCDEHJKL`ABCD|\\ &\verb|GILM`ABCEFILMN_ABDEFIKMO_@CDEGIJNP_ACDFGIKOQ_@BEFGIJPQR|\\
        $G_6$ & \verb|:KcA?hccUHApEOpwaIXCPBM`K@BGTK?cIXQ~|\\
        $G_7$ & \verb|:JbIOwbH?XGAGU?Pd_CPRF_COroAG`c~|\\
        $G_8$ & \verb|:JbIGgaF?XAoCKk@CN?HAf_CXbh_CPcI|\\
        \hline
    \end{tabular}
    \caption{Sparse6 representation for extremal graphs in Table \ref{tab:graphsbeta}.}
    \label{tab:sparse6beta}
\end{table}
\end{document}